\documentclass{article}

\usepackage[latin1]{inputenc}
\usepackage[]{xcolor}
\usepackage{graphicx}
\usepackage{subfig}
\usepackage{amsthm,amsmath,amssymb}
\usepackage{mathtools}
\usepackage{stmaryrd}
\usepackage[]{caption}
\usepackage{comment}
\usepackage{hyperref}
\usepackage{textcomp}

\newtheorem{thm}{Theorem}[section]
\newtheorem{prop}[thm]{Proposition}

\newtheorem{conj}[thm]{Conjecture}
\newtheorem{obs}[thm]{Observation}
\theoremstyle{plain}
\newtheorem{defi}[thm]{Definition}
\newtheorem{rem}[thm]{Remark}

\newcommand{\tref}[1]{Theorem~\textup{\ref{#1}}}
\newcommand{\pref}[1]{Proposition~\textup{\ref{#1}}}
\newcommand{\cref}[1]{Corollary~\textup{\ref{#1}}}

\newcommand{\rref}[1]{Remark~\textup{\ref{#1}}}
\newcommand{\dref}[1]{Definition~\textup{\ref{#1}}}
\newcommand{\fref}[1]{Figure~\textup{\ref{#1}}}
\newcommand{\sref}[1]{Section~\textup{\ref{#1}}}
\newcommand{\ssref}[1]{Subsection~\textup{\ref{#1}}}
\newcommand{\oref}[1]{Observation~\textup{\ref{#1}}}
\newcommand{\coref}[1]{Conjecture~\textup{\ref{#1}}}

\newcommand{\T}{\mathcal T}%Tableaux boisés
\newcommand{\TS}{\mathcal T^{Sym}}%Tableaux boisés symétriques
\newcommand{\s}{\sigma}
\newcommand{\E}{\mathbb E}
\newcommand{\Prob}{\mathbb P}

\let\leq=\leqslant
\let\geq=\geqslant
\let\ll=\llbracket
\let\rr=\rrbracket
\title{Occupied corners in tree-like tableaux}
\author{Patxi Laborde Zubieta}

\begin{document}
\maketitle
\begin{abstract}
    \begin{description}
 Tree-like tableaux are combinatorial objects that appear in a combinatorial understanding of the PASEP model from statistical mechanics. In this understanding, the corners of the Southeast border correspond to the locations where a particle may jump to the right. Such corners may be of two types: either empty or occupied. Our main result is the following: on average there is one occupied corner per tree-like tableau. We give two proofs of this, a short one which gives us a polynomial version of the result, and another one using a bijection between tree-like tableaux and permutations which gives us an additional information. Moreover, we obtain the same result for symmetric tree-like tableaux and we refine our main result to an equivalence class. Finally we present a conjecture enumerating corners, and we explain its consequences for the PASEP.
    \end{description}
\end{abstract}

\section{Introduction}
Tree-like tableaux are certain fillings of Young diagrams that were introduced by Aval, Boussicault and Nadeau in \cite{avbona13}. They are combinatorially equivalent to permutation tableaux and alternative tableaux \cite{SteinWil,Vien} which are counted by $n!$. Over the last years, these objects have been the subject of many papers \cite{Nad11,Burstein,CorNad}. The equilibrium state of the PASEP, an important model from statistical mechanics, can be described using these objects as shown in the papers \cite{CorWil_Markov,CorWil_Tableaux,CorWil10}. This model is a Markov chain where particles jump stochastically to the left or to the right by one site in a one-dimensional lattice of $n$ sites. We recall briefly its definition below.

The states of the PASEP are encoded by words in $\bullet$ and $\circ$ of length $n$, where $\circ$ corresponds to an empty site and $\bullet$ to a site with a particle. To perform a transition from a state, first choose uniformly at random one of the $n+1$ locations which are: at the left of the word, at the right of the word and between two sites. Then, with a probability defined beforehand, make a jump of particle if it is possible. \fref{fig:PASEP} gives an illustration of the PASEP model.
\begin{figure}[h!]
    \begin{center}
        \includegraphics[scale=.7]{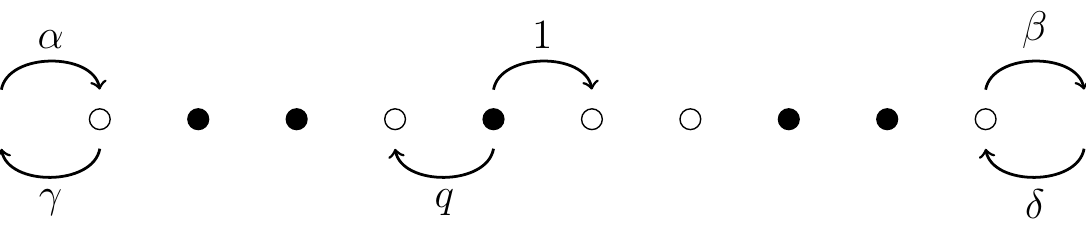}
    \caption{The PASEP model.}
    \label{fig:PASEP}
    \end{center}
\end{figure}
The link with tree-like tableaux appears when $\gamma=\delta=0$; there is a Markov chain over tree-like tableaux that projects to the PASEP \cite{CorWil_Markov}. The state to which a tree-like tableau projects is given by his Southeast border as follows: we travel along this border starting from the Southwest and without taking in account the first and the last step; an East step corresponds to $\bullet$ and a North step to $\circ$. In particular, the locations where a particle may jump to the right, \textit{i.e} the patterns $X\bullet\circ Y$, $\circ Y$ and $X\bullet$, coincides with \emph{corners}; while the locations where a particle may jump to the left, \textit{i.e} the patterns $X\circ\bullet Y$, since $\gamma=\delta=0$, corresponds to the \emph{inner corners}. The number of locations where a particle may jump to the right or to the left does not seem to have been studied yet.

In this paper, we explain in which way knowing the weighted average of the number of corners in tree-like tableaux of size $n$, gives us the weighted average of the number of locations where a particle may jump to the right or to the left in a state of the PASEP. Moreover, we conjecture there are $\frac{n+4}{6}$ corners per tree-like tableau, which implies $\frac{n+2}{3}$ jumping locations per PASEP state in the case $\alpha=\beta=q=1$. Quite naturally we made a computer exploration of the average number of corners in symmetric tree-like tableaux of size $2n+1$: it should be $\frac{4n+13}{12}$.
 We can distinguish two types of corners, the occupied corners which are filled with a point and the empty corners (cf. \dref{def:oc}). It appears that in average there is one occupied corner per tree-like tableau. We present two proofs of this result, a short one which gives us also a polynomial version, and another one using the bijection between tree-like tableaux and permutations that sends crossings to occurrences of the generalized pattern 2-31 (cf. \cite[Section~4.1]{avbona13}). The additional information given by the second proof is the proportion of occupied corners numbered $k$, where the numbering is given by the insertion algorithm. Using the same idea of the first proof, we show also that there is an occupied corner per symmetric tree-like tableaux. Moreover we can define an equivalence relation over tree-like tableaux (\ssref{sec:thm_oc_NA}) based on the position of the points in the diagram, we prove that in each equivalence class there is one occupied corner per tree-like tableau, it is a refinement of the previous result.

The article is organised as follows. \sref{sec:tlt} gives the definition of tree-like tableaux and introduces the insertion algorithm, $Insertpoint$. \sref{sec:enum_oc} is the main part of this paper: we give the definition of occupied corners and state our main result. Moreover, after giving the two different proofs, we extend this result to symmetric tree-like tableaux. We finish this section by a refinement of the main result that can be restated over lattice paths. \sref{sec:conj} describe more precisely the link between tree-like tableaux and the PASEP when $\delta=\gamma=0$, and states the two conjectures that enumerates corners in tree-like tableaux and symmetric tree-like tableaux.

\section{Tree-like tableaux}
\label{sec:tlt}
In this section we recall basic notions and tools about tree-like tableaux. All the details can be found in the article \cite{avbona13}.

\begin{defi}[Tree-like tableau]\label{def:tlt}
A \emph{tree-like tableau} is a filling of Young diagram with points inside some cells, with  three rules:
\begin{enumerate}
\item the top left cell has a point called the \emph{root point};
\item\label{def:tlt_cond_2} for each non root point, there is a point above it in the same column or to its left in the same row, but not both at the same time;
\item there is no empty row or column.
\end{enumerate}
\end{defi}

The \emph{size} of a tree-like tableau is its number of points. We denote by $\T_n$ the set of the tree-like tableaux of size $n$. In a tree-like tableau, we call \emph{border edges}, the edges of the Southeast border. A tree-like tableau of size \begin{math}n\end{math} has \begin{math}n+1\end{math} border edges, we index them from 1 to \begin{math}n+1\end{math} as it is done in \fref{fig:tlt_ex}. The border edge numbered by \begin{math}i\end{math} in a tree-like tableau \begin{math}T\end{math} is denoted by \begin{math}e_i(T)\end{math}. In the rest of the article, when there is no ambiguity, the tableau $T$ might be omitted in all the notations.

\begin{figure}[ht]
    \begin{center}
        \includegraphics[scale=.6]{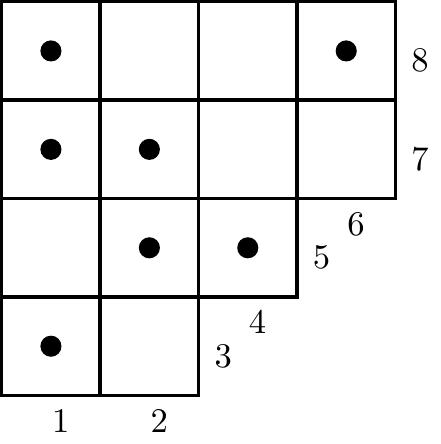}
        \caption{A tree-like tableau of size 7.}
        \label{fig:tlt_ex}
    \end{center}
\end{figure}

The set of tree-like tableaux has an inductive structure given by an insertion algorithm called $Insertpoint$ which constructs a tree-like tableau of size \begin{math}n+1\end{math} from a tree-like tableau of size \begin{math}n\end{math} and the choice of one of its border edges. We briefly present the algorithm in order for the article to be self-contained. A more detailed presentation is given in \cite{avbona13}.\newline
The \emph{special point} plays an important role in $Insertpoint$, it is the right-most point among those at the bottom of a column. We denote by \begin{math}sp(T)\end{math} the index of the horizontal border edge under the special point of a tree-like tableau \begin{math}T\end{math}. In the figures of this article, the special point might be indicated by a square around it. The second notion we need is the \emph{ribbon}, it is a connected set of empty cells (with respect to adjacency) containing no \begin{math}2\times 2\end{math} squares. Now we can introduce the insertion algorithm.

\begin{defi}[Insertpoint]
Let \begin{math}T\end{math} be a tree-like tableau of size \begin{math}n\end{math} and \begin{math}e_i\end{math} one of its border edges. First, if \begin{math}e_i\end{math} is horizontal (resp. vertical) we insert a row (resp. column) of empty cells, just below (resp. to the right of) $e_i$, starting from the left (resp. top) border of $T$ and ending below (resp. to the right of) \begin{math}e_i\end{math}. Moreover we put a point in the right-most (resp. bottom) cell. We obtain a tree-like tableau \begin{math}T'\end{math} of size \begin{math}n+1\end{math}. Then, depending on the relative ordering of \begin{math}i\end{math} and \begin{math}sp(T)\end{math}, we define \begin{math}Insertpoint(T,e_i)\end{math} as follows.
\begin{enumerate}
\item If \begin{math}i\geq sp\end{math}, then \begin{math}Insertpoint(T,e_i)= T'\end{math};
\item otherwise, we add to \begin{math}T'\end{math} a ribbon along the border, from the new point to the special point of \begin{math}T\end{math}. This new tree-like tableau will be \begin{math}Insertpoint(T,e_i)\end{math}.
\end{enumerate}
\end{defi}
An example of the two possible insertions is given in \fref{Fig:AlgIns} where the cells of the new row/column are shaded and the cells of the ribbon are marked by crosses.
\begin{figure}[h!]
    \begin{center}
        \captionsetup[subfigure]{width=84pt,format=hang,singlelinecheck=false}
        \subfloat[A tree-like tableau.]{\includegraphics[scale=.4]{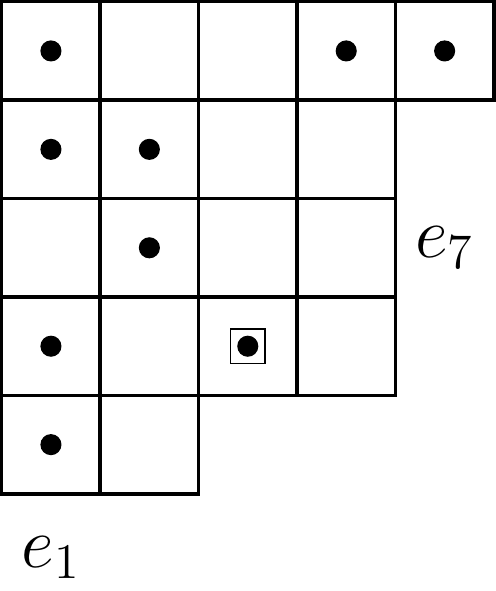}}
        \hfil\hfil\hfil
        \subfloat[Insertion at $e_1$.]{\includegraphics[scale=.4]{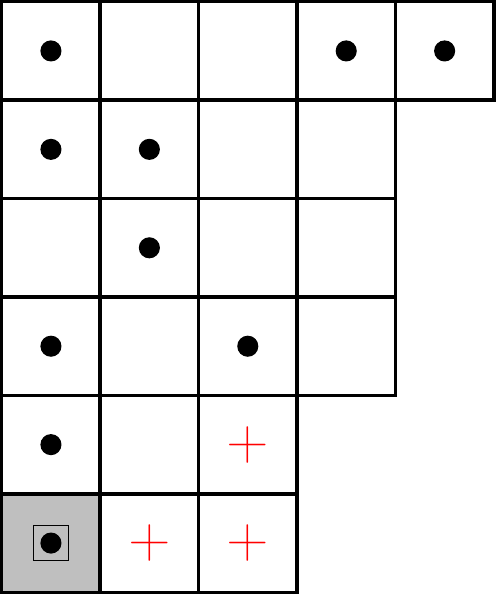}}
        \hfil\hfil\hfil
        \subfloat[Insertion at $e_7$.]{\includegraphics[scale=.4]{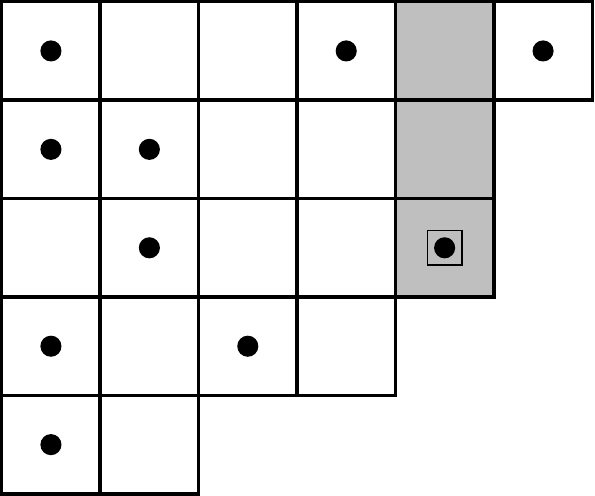}}
        \caption{Insertion algorithm applied to $T$ at $e_1$ and $e_7$.}
        \label{Fig:AlgIns}
    \end{center}
\end{figure}
\begin{rem}\label{rem:index_sp}
 We notice that the new point is the special point of the new tableau. In addition, the index of the horizontal border edge under the new point is equal to the one we chose during the algorithm, in other words:
$$sp(Insertpoint(T,i))=i.$$
\end{rem}

\section{Enumeration of occupied corners}
\label{sec:enum_oc}
As explained in the introduction, \emph{corners} (see \dref{def:oc}) are interesting because in the PASEP model they correspond to the locations where a particle may jump to the right. Our main result is that in $\T_n$, on average there is one \emph{occupied corner} per tree-like tableaux. It is a nice and surprising property, and a first step in the study of unrestricted corners. In this section, we show this new result in two\footnote{In the unpublished note \cite{LZ}, we present a third proof based on two mirror insertion algorithms.} different ways. The first proof (\ssref{sec:shortproof}) is the shortest one and gives us a polynomial version of the result. The second proof (\ssref{sec:enum_proof}) is interesting because it tells us in which proportion the point number $k$ is in a corner, where the numbering is induced by the insertion algorithm. In \ssref{sec:sym}, we extend the result to symmetric tree-like tableaux, again with a polynomial version. Finally, in the last subsection, we define an equivalence relation over the tree-like tableaux and we show that in each equivalence class, there is one occupied corner per tree-like tableaux.
\subsection{Main result}
\label{sec:main_result}

First of all, let us define occupied corners.
\begin{defi}[Occupied corner]\label{def:oc}
In a tree-like tableau $T$, the \emph{corners} are the cells for which the bottom and the right edges are border edges. We say that a corner is \emph{occupied}, if it contains a point. We denote by $oc(T)$ the number of occupied corners of $T$, we extend this notation to a subset $X$ of $\T_n$ as follows, \begin{math}oc(X)=\sum_{T\in X}oc(T).\end{math}
\end{defi}
\fref{fig:def_oc} gives us an example of a tree-like tableau $T$ for which $oc(T)=2$.
\begin{figure}[h!]
    \begin{center}
        \includegraphics[scale=.4]{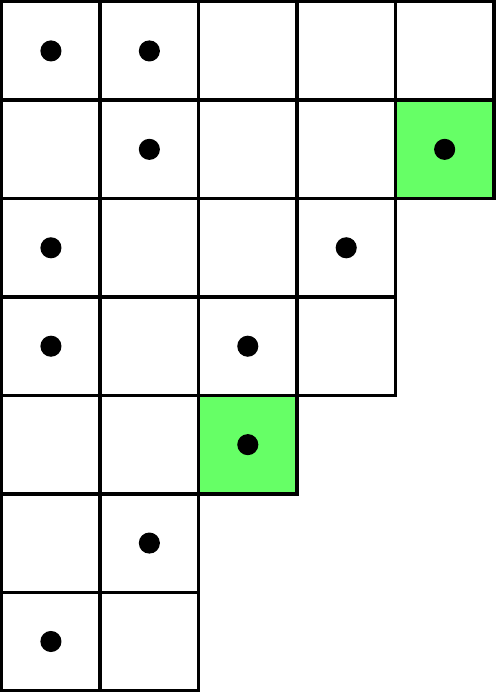}
        \caption{A tree-like tableau with 4 corners, 2 of which are occupied.}
        \label{fig:def_oc}
    \end{center}
\end{figure}

\begin{thm}\label{thm:main_result}
The number of occupied corners in the set of tree-like tableaux of size $n$ is given by:
\begin{displaymath}
oc(\T_n)=n!.
\end{displaymath}
\end{thm}
In other words, on average there is one occupied corner per tree-like tableau. The reader may check \tref{thm:main_result} for the case $n=3$ using \fref{fig:thm_case_3}.
\begin{figure}[h!]
    \begin{center}
        \captionsetup[subfigure]{margin=-4pt}
        \subfloat[$oc=2$]{\includegraphics[scale=.5]{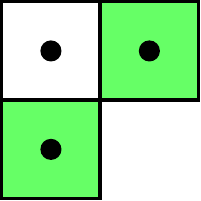}}
        \hfil
        \captionsetup[subfigure]{margin=-12pt}
        \subfloat[$oc=1$]{\includegraphics[scale=.5]{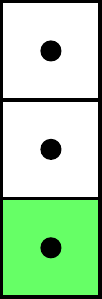}}
        \hfil
        \captionsetup[subfigure]{margin=-5pt}
        \subfloat[$oc=1$]{\includegraphics[scale=.5]{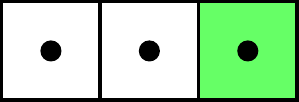}}
        \hfil
        \subfloat[$oc=1$]{\includegraphics[scale=.5]{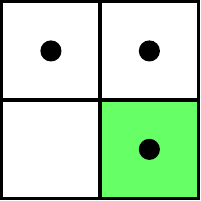}}
        \hfil
        \subfloat[$oc=1$]{\includegraphics[scale=.5]{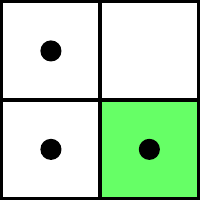}}
        \hfil
        \subfloat[$oc=0$]{\includegraphics[scale=.5]{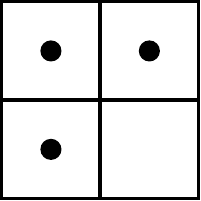}}
        \hfil
        \caption{\tref{thm:main_result} for n=3.}
        \label{fig:thm_case_3}
    \end{center}
\end{figure}

\subsection{Short proof}\label{sec:shortproof}
We put in bijection the set of occupied corners in $\T_n$ with the set of pairs composed of a tree-like tableau of size $n-1$ and an integer of $\ll 1,n \rr$. Let $OC$ be an occupied corner in $\T_n$, let $T$ be the tree-like tableau corresponding to $OC$ and let $i$ be the index of the horizontal border edge of $OC$. Among the row and the column containing $OC$ there is one which has no other point than the one of $OC$. In the case where it is the row (resp. column), we remove this row (resp. column) and we call $T'$ the tree-like tableau of size $n-1$ we obtain. In $T'$ the border edge of index $i$ is horizontal (resp. vertical). From the pair $(T',i)$ we can construct back $OC$ by inserting a row (resp. a column) at $e_i(T')$ and putting a point in the right-most (resp. the bottom) cell. The \fref{fig:thm_easy} gives us an illustration of the bijection. The conclusion follows since the set of pairs composed of a tree-like tableau of size $n-1$ and an integer of $\ll 1,n \rr$ has cardinality $(n-1)!\times n=n!$.
\begin{figure}[h!]
    \begin{center}
        \subfloat[A tree-like tableau with two corners $c_1$ and $c_2$.]{\includegraphics[scale=.6]{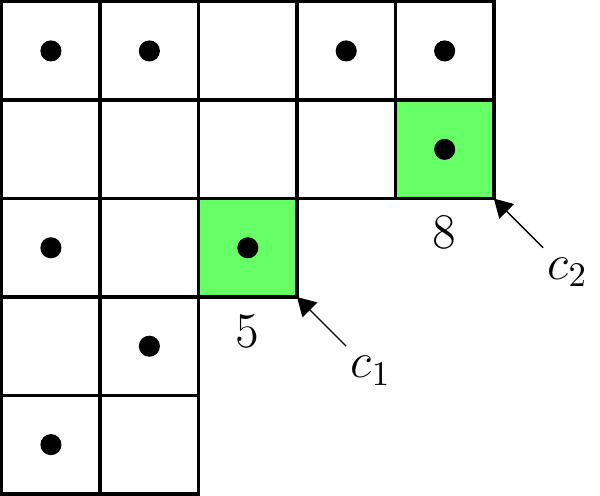}}
        \hfill
        \subfloat[The pair associated to $c_1$.]{\includegraphics[scale=.6]{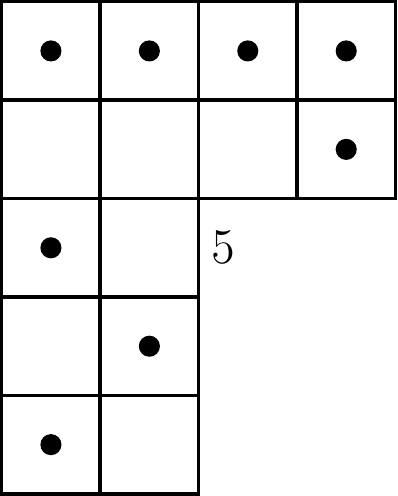}}
        \hfill
        \subfloat[The pair associated to $c_2$.]{\includegraphics[scale=.6]{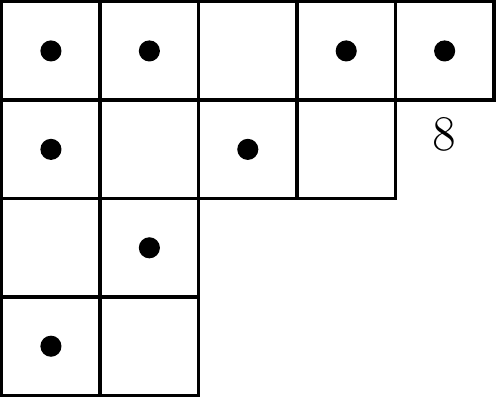}}
        \hfill
        \caption{Bijective proof of \tref{thm:main_result}}
        \label{fig:thm_easy}
    \end{center}
\end{figure}

We can obtain a polynomial version of the main result. Let $P_n(x)=\sum_{T\in\T_n}x^{oc(T)}=\sum_{k\geq 0}a_{n,k}x^k$, the coefficient $a_{n,k}$ corresponds to the number of tree-like tableaux of size $n$ with $k$ occupied corners. In order to determine a recurrence relation defining these polynomials, we start by finding a recurrence relation for the coefficients $a_{n,k}$. We do it by analysing how the number of occupied corners varies in the previous bijection during the addition of an occupied corner. Let $T'$ be a tree-like tableau of size $n-1$, we denote by $k$ the number of occupied corners of $T'$. Applying the previous bijection, if we choose one of the $2k$ border edges of the occupied corners of $T'$, we obtain a tree-like tableau $T$ with $k$ occupied corners, else $T$ has $k+1$ occupied corners. By performing the bijection over all the pairs of $\T_{n-1}\times\ll 1,n \rr$, a tree-like tableau $T$ of size $n$ is obtained $oc(T)$ times. For $n\geq 2$ and $k\geq 1$, we obtain the following recurrence relation:
$$ka_{n,k}=2ka_{n-1,k}+(n-2(k-1))a_{n-1,k-1}.$$
This can be translated as a recurrence relation over the polynomials $P_n$ as follows:
$$xP_n'=2xP_{n-1}'+nxP_{n-1}-2x^2P_{n-1}'.$$
Which is equivalent to :
\begin{equation}\label{eq:main_result_pol_version}
P_n'=nP_{n-1}+2(1-x)P_{n-1}'
\end{equation}
If we set $P_0=1$, the Equation~\eqref{eq:main_result_pol_version} is also true for $n=1$. The sequence of polynomials $(P_n)_{n\geq 0}$ is uniquely defined by:
$$\left\{
\begin{array}{l}
\forall n\geq 1,\text{ }P_n'=nP_{n-1}+2(1-x)P_{n-1}'\text{ and } P_n(1)=n!;\\
P_0=1.
\end{array}
\right.
$$
We recover the main result by noticing that the number of occupied corners in $\T_n$ is equal to $P_n'(1)$. By Equation~\eqref{eq:main_result_pol_version}, $P_n'(1)=nP_{n-1}(1)=n!$ since $P_{n-1}(1)=|\T_{n-1}|=(n-1)!$. To illustrate, we give the first terms of the sequence $(P_n)_{n\geq0}$:
\begin{flalign*}
P_0&=1\\
P_1&=x\\
P_2&=2x\\
P_3&=x^2+4x+1\\
P_4&=6x^2+12x+6.\\
   &\ \ \vdots\\
P_{10}&=720x^5 + 21600x^4 + 188640x^3 + 748800x^2 + 1475280x + 1193760\\
\end{flalign*}

\begin{rem}
The number of tree-like tableaux of size $n$ with 0 occupied corner is given by $P_n(0)$.  By computer exploration using the open-software Sage \cite{sage}, we found that the sequence $(P_n(0))_{n\geq0}$ might be equal to the sequence \cite[A184185]{oeis} which counts the number of permutations of $\{1,\ldots,n\}$ having no cycles of the form $(i,i+1,i+2,\ldots,i+j-1)$ with $j\geq1$. To this point, we were not able to prove this fact.
\end{rem}
In order to estimate the dispersion of the statistic of occupied corners, we will compute the variance $\s^2$. We have the formula $\sigma^2=\frac{1}{n!}(\sum_{T\in\T_n}oc(T)^2)-1$, hence,
\begin{flalign*}
\sigma^2&=\frac{1}{n!}(P_n''(1)+P'_n(1))-1\\
        &=\frac{1}{n!}((n-2)P_{n-1}'(1)+P_n'(1))-1\\
        &=\frac{1}{n!}((n-2)(n-1)!+n!)-1\\
        &=\frac{n-2}{n}
\end{flalign*}
Hence, the variance tends to 1 when $n$ goes to infinity.

\subsection{Enumerative proof using a bijection with permutations}\label{sec:enum_proof}
Let $T$ be a tree-like tableau of size $n\geq2$. There is a unique way to obtain $T$ by performing consecutive insertions starting with the tree-like tableau of size 1. Since $T$ is of size $n$, we do $n-1$ insertions. Hence, we can uniquely encode $T\in\T_n$ by the sequence of the indices chosen during the insertion algorithm: \begin{math}(m_2(T),\ldots,m_n(T))\end{math} where \begin{math}1\leq m_k(T)\leq k\end{math} for \begin{math}k\in\ll 2,n \rr\end{math}. For $k\in\ll 2,n \rr$, we denote by $p_k(T)$ the point added during the $(k-1)$-th insertion. Let us recall that (cf. \rref{rem:index_sp}) when $p_k$ is inserted, the index of the border edge under it, is $m_k$. By convention, $m_1(T)=1$ and $p_1(T)$ corresponds to the point of the unique tree-like tableau of size 1. For example, the tree-like tableau of \fref{fig:tlt_ex} is encoded by $(1,1,3,2,2,1,4)$.

\begin{obs}\label{obs:oc_code}
Let $k\in\ll 2,n \rr$, the point $p_k$ is in a corner if, and only if, the following two conditions occur:
\begin{itemize}
\item $m_k\geq m_{k-1}$, \textit{i.e} if no ribbon is added during its insertion;
\item $m_j>m_k+1$ for all $k<j\leq n$, \textit{i.e} if after $p_k$ all the other points are inserted to the Northeast of the vertical border edge $e_{m_k+1}$ ($p_k$ is not covered by a ribbon).
\end{itemize}
\end{obs}
In order to get a statistic over permutations which is in bijection with the occupied corners and keep the information of the index $k$, we introduce the bijection $\phi$, defined in \cite[Section~4.1]{avbona13}, between tree-like tableaux and permutations. Let $T$ be a tree-like tableau of size $n$ and $(m_1,\ldots,m_n)$ its encoding, $\s=\phi(T)$ is the unique permutation with the non inversion table equal to $(m_1-1,\ldots,m_n-1)$. It can be computed algorithmically in the following way. First, set $\s(n)$ to $m_n$. Suppose now we have defined $\s(n),\ldots,\s(i+1)$. If we take the following notation,
\begin{displaymath}
\{x_1<x_2<\cdots<x_{i}\}=\ll 1,n \rr\setminus\{\s(i+1),\ldots,\s(n)\}
\end{displaymath}
we take $\s(i)=x_{m_i}$. In other words, consider the word $12\cdots n$, remove the $m_n$-th letter, then the $m_{n-1}$-th one, and so on, until getting the empty word. Then, concatenate the letters in the order they got removed, the first removed one being at the end of the word and the last removed one at the beginning. This way, we obtain $\s$ as a word. Let us give an example of this bijection, consider the tree-like tableau $T$ of \fref{fig:tlt_ex} encoded by $(1,1,3,2,2,1,4)$, we have $\phi(T)=6275314$. Now we can translate \oref{obs:oc_code} as a result over permutations.

\begin{prop}\label{prop:stat_perm}
Let $T\in\T_n$ and $\s=\phi(T)$, then for $k\in\ll 2,n \rr$, $p_k$ is in a corner of $T$ if, and only if, the following two conditions occur:
\begin{enumerate}
\item\label{prop:stat_perm_cond1} $\s(k-1)=\s(k)+1$ or $\s(k-1)<\s(k)$;
\item\label{prop:stat_perm_cond2} $\s(j)>\s(k)+1$ for all $k<j\leq n$.
\end{enumerate}
\end{prop}
\begin{proof}
We will first prove that if $p_k$ is in a corner of $T$, the index $k$ satisfies the conditions \ref{prop:stat_perm_cond1} and \ref{prop:stat_perm_cond2}. We use the characterisation of \oref{obs:oc_code} to prove the first implication. If $k=n$, we only need to prove condition \ref{prop:stat_perm_cond1}, there are two possibilities, if $m_n=n$, then $\s(n)=n$ and $\s(n-1)<\s(n)$, otherwise, in the construction process of $\s$, we choose $\s(n)=m_n$ hence for $\s(n-1)$ we choose the $m_{n-1}$-th letter of the word $1\cdots (m_n-1)(m_n+1)\cdots n$, and since $m_{n-1}\leq m_n$ we have $\s(n-1)\in\{1,\ldots,m_n-1,m_n+1\}$. In the case $k\neq n$, we know that $m_j>m_k+1$ for all $k<j\leq n$, hence during the construction of those $\s(j)$, the subword $1\cdots m_k(m_k+1)$ of the initial word $1\cdots n$ is never modified, in particular, $\s(j)>m_k+1$. Since $m_k=n$ implies $k=n$, the subword $1\cdots m_k(m_k+1)$ is well defined. Then, we have $\s(k)=m_k$ hence condition \ref{prop:stat_perm_cond2} is satisfied. Moreover, the $m_k$ first letters of $1\cdots n$, after removing $\s(n),\ldots,\s(k)$, are $1,2,\cdots,m_k-1,m_k+1$ thus $\s(k-1)\in\{1,\ldots,m_k-1,m_k+1\}$ as $m_{k-1}\leq m_k$, which prove condition \ref{prop:stat_perm_cond1}.

Conversely, let's suppose that $\s$ satisfies conditions 1 and 2. We will show that the encoding $(m_1,\ldots,m_n)$ of the corresponding tree-like tableau satisfies the conditions of \oref{obs:oc_code}. As $(m_1-1,\ldots,m_n-1)$ is the non-inversion table of $\s$, for $j\in\ll 1,n \rr$, $m_j-1$ corresponds to the number of $\s(i)$ such that $i<j$ and $\s(i)<\s(j)$. First, we notice that $\s(k)+1$ is at the left of $\s(k)$ by the condition \ref{prop:stat_perm_cond2}. For $j>k$, in the word $\s$, at the left of $\s(j)$ there are $\s(k)+1$, $\s(k)$ and all the integers that are at the left of $\s(k)$ and smaller than $\s(k)$. Hence $m_j-1\geq 2+m_k-1$ which is equivalent to $m_j>m_k+1$. Condition \ref{prop:stat_perm_cond2} implies that if for $j<k-1$ we have: $\s(j)<\s(k-1)$ implies $\s(j)<\s(k)$. Thus, $m_{k-1}\leq m_k$.
\end{proof}
Thanks to this last interpretation, we obtain a result refining \tref{thm:main_result}.

\begin{prop}
For $k\in\ll 2,n \rr$, the number of tree-like tableaux $T$ of size $n$ having $p_k$ in a corner is
\begin{equation}\label{eq:nb_pk_in_corner}
\frac{n!}{(n-k+2)(n-k+1)}+\left\lbrace\begin{array}{cl} (n-1)!&\text{if }k=n;\\ 0&\text{otherwise.}\\\end{array}\right.
\end{equation}
\end{prop}
\begin{proof}
In order to show this, we enumerate the permutations for which $k$ satisfies the conditions of \pref{prop:stat_perm}. First if $k=n$, only condition \ref{prop:stat_perm_cond1} matters. Thus, if $\s(n)=n$ we can choose the other value of $\s$ as we want, hence we have $(n-1)!$ possibilities to choose $\s$. Else, if we denote by $i$ the integer $\s(n)\leq n-1$, we have $i$ possibilities for $\s(n-1)$, namely $1,\ldots,i-1,i+1$ and, there is no other constrain in $\s$, hence we have $i(n-2)!$ such permutations. Doing the sum we obtain,
$$\sum_{i=1}^{n-1}i(n-2)!+(n-1)!=\frac{n!}{2}+(n-1)!,$$
which is equal to Equation~\eqref{eq:nb_pk_in_corner} when $k=n$.
Suppose now $k\neq n$, let $\s(k)=i$. Condition \ref{prop:stat_perm_cond2} tells us that the integers $\s(k+1),\ldots,\s(n)$ are inside $\ll i+2,n \rr$, hence $n-k\leq n-i-1$ which is equivalent to $i\leq k-1$. There are $\binom{n-i-1}{n-k}$ choices for those integers and $(n-k)!$ ways of ordering them. Then, we have $i$ choices for $\s(k-1)$, namely $1,\ldots,i-1$ and $i+1$. Finally, to define $\s(1),\ldots,\s(k-2)$, we have $(k-2)!$ possibilities. By summing we obtain:
\begin{equation}\label{eq:nb_pk_in_corner_non_simplifie}
(n-k)!(k-2)!\sum_{i=1}^{k-1}i\binom{n-i-1}{n-k}.
\end{equation}
Let $S_i=\sum_{l=i}^{k-1}\binom{n-l-1}{n-k}$ for $i\in\ll 1,k-1 \rr$ and $S_k=0$. We have that $S_i=\sum_{j=n-k}^{n-i-1}\binom{j}{n-k}=\binom{n-i}{n-k+1}$. We can simplify Equation~\eqref{eq:nb_pk_in_corner_non_simplifie} as follows:
\begin{displaymath}
\sum_{i=1}^{k-1}i\binom{n-i-1}{n-k}=\sum_{i=1}^{k-1}i(S_i-S_{i+1})=\sum_{i=1}^{k-1}S_i=\sum_{i=1}^{k-1}\binom{n-i}{n-k+1}=\binom{n}{n-k+2}
\end{displaymath}
We get the desired result since:
\begin{displaymath}
(n-k)!(k-2)!\binom{n}{n-k+2}=\frac{n!}{(n-k+2)(n-k+1)}
\end{displaymath}
\end{proof}
We can check we have $n!$ occupied corners by summing Equation~\eqref{eq:nb_pk_in_corner} for $k$ in $\ll 2,n \rr$:

\begin{flalign*}
&=(n-1)!+n!\sum_{k=2}^n\frac{1}{(n-k+2)(n-k+1)}\\
&=(n-1)!+n!\sum_{k=2}^n\left(\frac{1}{(n-k+1)}-\frac{1}{(n-k+2)}\right)\\
&=(n-1)!+n!\left(1-\frac{1}{n}\right)\\
&=(n-1)!+(n-1)(n-1)!\\
&=n!.
\end{flalign*}
\subsection{Extension to symmetric tree-like tableaux}\label{sec:sym}
Symmetric tree-like tableaux were introduced in \cite[Section~2.2]{avbona13}, they are the tree-like tableaux which are invariant by an axial symmetry with respect to their main diagonal. They are in bijection with symmetric alternative tableaux from \cite[Section~3.5]{Nad11} and type B permutation tableaux from \cite[Section~10.2.1]{LamWil}. The root point is the only point of the main diagonal of a symmetric tree-like tableau. Indeed, if there was another one, it would have a point to its left and above him, or neither of them which would contradict condition \ref{def:tlt_cond_2} of \dref{def:tlt}. Hence such a tree-like tableau has odd size. We denote by $\TS_{2n+1}$ the set of symmetric tree-like tableaux of size $2n+1$, there are $2^nn!$ of them (\cite[Section~2.2]{avbona13}).

\begin{thm}\label{thm:main_result_sym}
The number of occupied corners in the set $\TS_{2n+1}$ is given by
\begin{displaymath}
oc(\TS_{2n+1})=2^nn!.
\end{displaymath}
\end{thm}
\begin{proof}
Similarly to \ssref{sec:shortproof}, we will put in bijections the set of occupied corners of $\TS_{2n+1}$ with the set of triplets composed of a symmetric tree-like tableau of size $2n-1$, an integer of $\ll 1,n \rr$ and an element $\rho$ of $\{a,b\}$. Let $OC$ be an occupied corner in $\TS_{2n+1}$, let $T$ be the tree-like tableau corresponding to $OC$. If $OC$ is below the main diagonal of $T$, we set $\rho$ to $b$ and we denote by $i$ the index of the horizontal border edge of $OC$. Else, we consider the occupied corner symmetric to $OC$ and we denote by $i$ the index of its horizontal border edge, moreover we set $\rho$ to $a$. We denote $T'$ the symmetric tableau obtained after removing the nearly empty row or column in which $OC$ is and its symmetric. We associate to $OC$ the triplet $(T',i,\rho)$. From the pair $(T',i)$ we construct back the tree-like tableau $T$ by inserting a row or a column at $e_i(T')$ and putting a point in the right-most or the bottom cell, and doing also the symmetric. We constructed two new occupied corners, $OC$ correspond to the one below the main diagonal if $\rho=b$ and to the one above the main diagonal if $\rho=a$. Finally $oc(\TS_{2n+1})=2^{n-1}(n-1)!\times n\times 2=2^nn!$.
\end{proof}
Let $Q_n=\sum_{T\in\TS_{2n+1}}x^{oc(T)}=\sum_{k\geq 0}b_{n,k}x^{2k}$. In the case of symmetric tree-like tableaux, we obtain the following recurrence relation. For $n\geq2$ and $k\geq1$, we have:
$$2k\cdot b_{n,k}=2\left[2k\cdot b_{n-1,k}+(n-2(k-1))b_{n-1,k-1}\right].$$
This can be translated as a recurrence relation over the polynomials $Q_n$ as follows:
$$xQ_n=2\left[xQ_{n-1}'+nx^2Q_{n-1}-x^3Q_{n-1}'\right]$$
Which is equivalent to:
\begin{equation}\label{eq:main_result_sym_pol_version}
Q_n'=2nx\cdot Q_{n-1}+2(1-x^2)Q_{n-1}'
\end{equation}
If we set $Q_0=1$, the Equation~\eqref{eq:main_result_sym_pol_version} is also true for $n=1$. The sequence of polynomials $(Q_n)_{n\geq 0}$ is uniquely defined by:
$$\left\{
\begin{array}{l}
\forall n\geq 1,\text{ }Q_n'=2nx\cdot Q_{n-1}+2(1-x^2)Q_{n-1}'\text{ and } Q_n(1)=2^nn!;\\
Q_0=1.
\end{array}
\right.
$$
To illustrate, we give the first terms of the sequence $(Q_n)_{n\geq0}$:
\begin{flalign*}
Q_0&=1\\
Q_1&=x^2+1\\
Q_2&=4x^2+4\\
Q_3&=2x^4+20x^2+26\\
   &\ \ \vdots\\
Q_9&=384x^{10} + 55680x^8 + 1386240x^6 + 13566720x^4 + 61380480x^2 + 109405056.\\
\end{flalign*}

As we did for occupied corners in $\T_n$, we can compute the variance. It is equal to $\frac{n-1}{n}$, hence we can draw the same conclusion, it tends to 1 when $n$ goes to infinity.

\subsection{A refinement of the main result}\label{sec:thm_oc_NA}
A \emph{non-ambiguous class} is a set of tree-like tableaux which have their points arranged in the same way (\fref{fig:non_ambiguous_class}). These classes are linked with non-ambiguous trees introduced in \cite{abbs14}. Indeed, they can be described as the set of tree-like tableaux which have the same underlying non-ambiguous tree.
\begin{thm}\label{thm:oc_NA}
In a non-ambiguous class, on average there is one occupied corner per tree-like tableau.
\end{thm}
\fref{fig:non_ambiguous_class} gives an example of a non-ambiguous class of size 5 with 5 occupied corners in total. It is a refinement of \tref{thm:main_result} since we can partition $\T_n$ in non-ambiguous classes.
\begin{figure}[h!]
    \begin{center}
    \captionsetup[subfigure]{margin=-9pt,format=hang,singlelinecheck=false}
    \subfloat[$oc=3$]{\includegraphics[scale=.4]{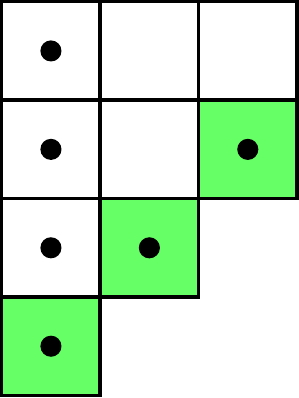}}
    \hfil
    \subfloat[$oc=1$]{\includegraphics[scale=.4]{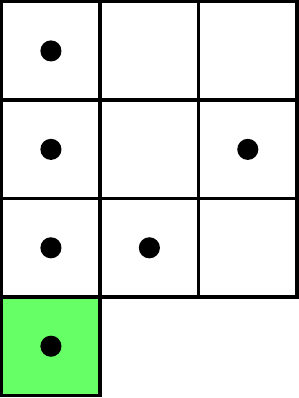}}
    \hfil
    \subfloat[$oc=1$]{\includegraphics[scale=.4]{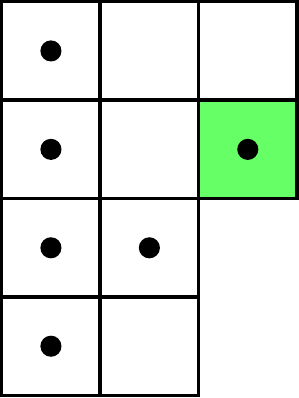}}
    \hfil
    \subfloat[$oc=0$]{\includegraphics[scale=.4]{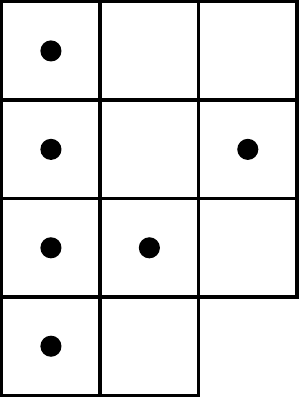}}
    \hfil
    \subfloat[$oc=0$]{\includegraphics[scale=.4]{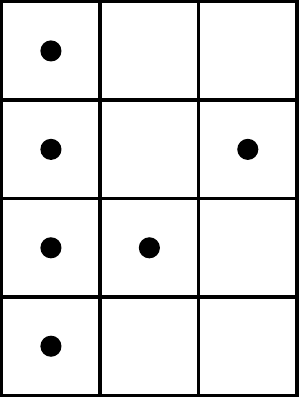}}
    \end{center}
    \caption{A non-ambiguous class in $\T_5$}
    \label{fig:non_ambiguous_class}
\end{figure}
We first show that it can be restated over paths. For each non-ambiguous class, we choose a canonical representative: the only tree-like tableau whose corners are all occupied. The other tree-like tableaux of the class are obtained by choosing a lattice path, with steps East and North, weakly below the Southeast border of the canonical representative. \fref{fig:correspondence_NA_paths} gives us an example of this correspondence, only the points which are in the corners of the canonical representative are drawn.
\begin{figure}[h!]
    \begin{center}
    \captionsetup[subfigure]{margin=-9pt,format=hang,singlelinecheck=false}
    \subfloat[A canonical representative of a non-ambiguous class.]{\includegraphics[scale=.35]{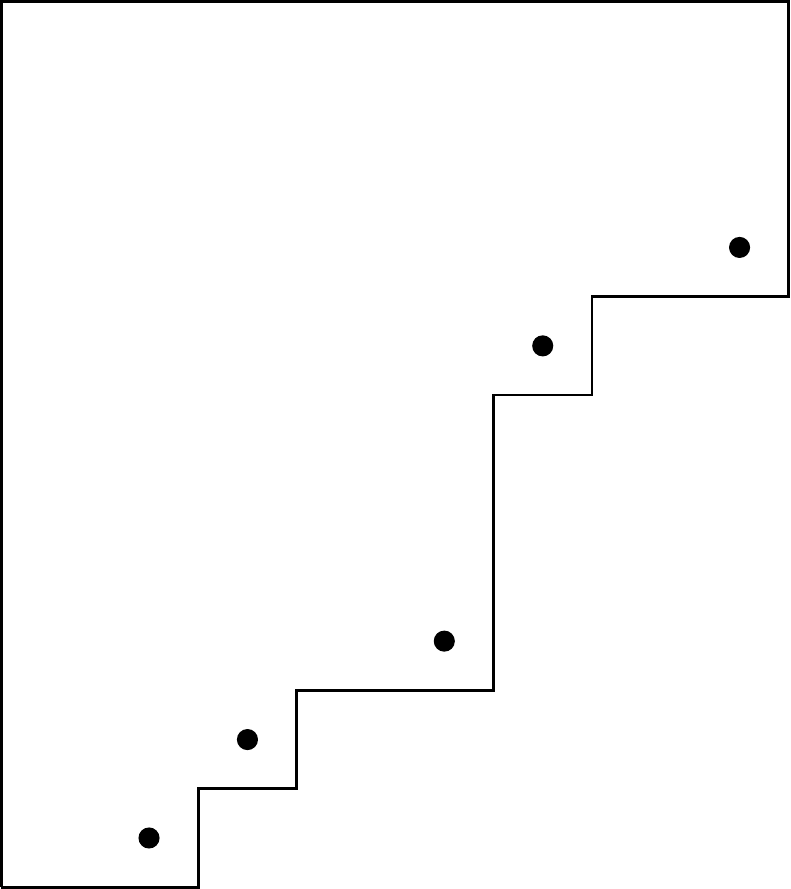}}
    \hfil
    \subfloat[An other element of a non-ambiguous class.]{\includegraphics[scale=.35]{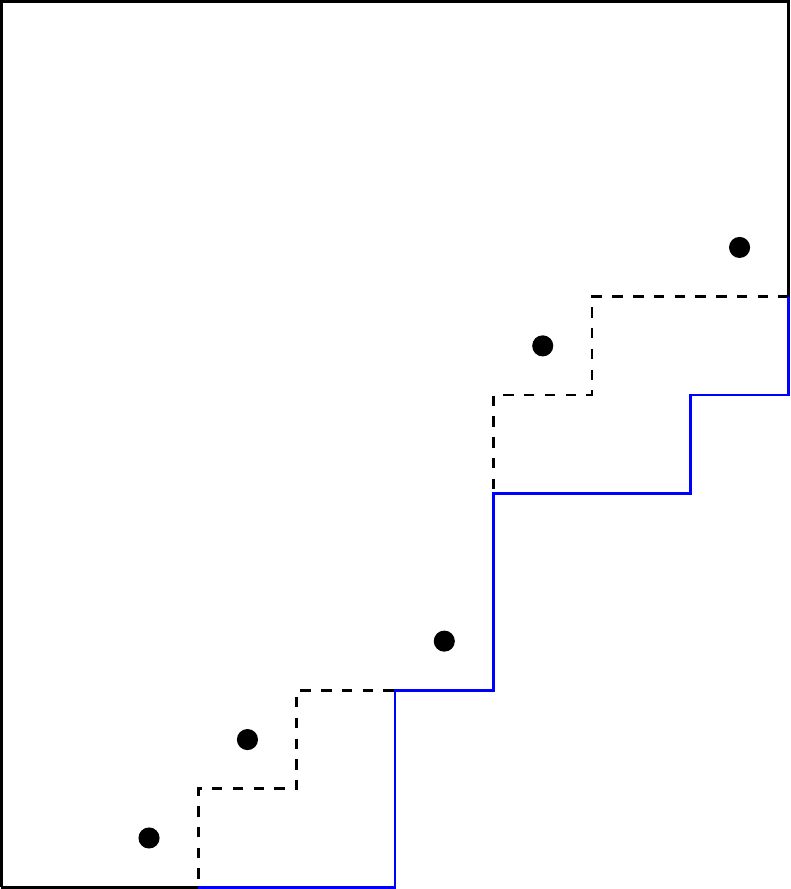}}
    \end{center}
    \caption{Correspondence between non-ambiguous class and paths.}
    \label{fig:correspondence_NA_paths}
\end{figure}
In this study, the only thing that matters is the subpath of the Southeast border of the canonical representative which is between the Southwest-most corner and the Northeast-most corner. In the rest of the section, we will only consider lattice paths with North and East steps. Let $P$ be a lattice path starting with a corner, \textit{i.e} an East step followed by a North step, and finishing with a corner. For a lattice path $P'$ weakly below $P$ we denote by $cc(P')$ the number of corners $P'$ and $P$ have in common. Then we can restate \tref{thm:oc_NA} with the following equality:
\begin{displaymath}
|\{\text{lattice paths weakly below }P\}|=\sum_{P'\text{ weakly below }P}cc(P').
\end{displaymath}
\fref{fig:paths_class} shows us the corresponding example in terms of paths of \fref{fig:non_ambiguous_class}.
\begin{figure}[h!]
    \begin{center}
    \captionsetup[subfigure]{margin=-9pt,format=hang,singlelinecheck=false}
    \subfloat[$oc=3$]{\includegraphics[scale=.4]{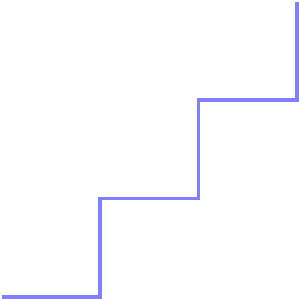}}
    \hfil
    \subfloat[$oc=1$]{\includegraphics[scale=.4]{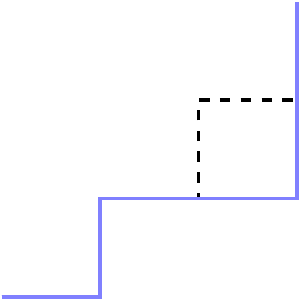}}
    \hfil
    \subfloat[$oc=1$]{\includegraphics[scale=.4]{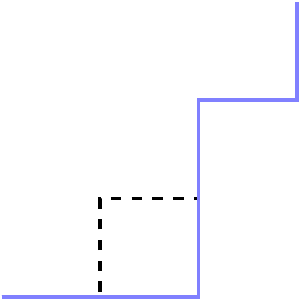}}
    \hfil
    \subfloat[$oc=0$]{\includegraphics[scale=.4]{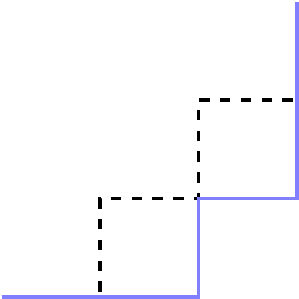}}
    \hfil
    \subfloat[$oc=0$]{\includegraphics[scale=.4]{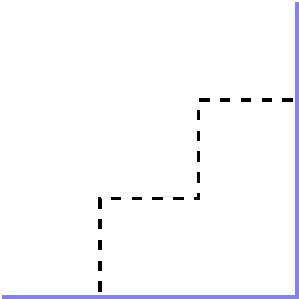}}
    \end{center}
    \caption{A non-ambiguous class in $\T_5$}
    \label{fig:paths_class}
\end{figure}
In order to show this, we will put in bijection each path $P'$ weakly below $P$ with $cc(P')=k\geq 1$, with $k$ paths weakly below $P$. Let $c_1,\ldots,c_k$ be the $k$ corners of $P'$ which are in common with $P$. For each $i\in\ll 1,k \rr$, if $c_i$ corresponds to the Northeast-most corner of $P$, we associate to $c_i$ the path $P'$ itself, else the path obtained from $P'$ by shifting by one step to the South the subpath of $P'$ which is at the Northeast of $c_i$. Conversely, if we consider a path $P'$ weakly below $P$, to find the corner to which it is associated, first we spot the Northeast-most East step $P'$ and $P$ have in common, we denote this step by $s$. Then, if $s$ corresponds to the last East step of $P$, we associate to $P'$ his Northeast-most corner. Else, we shift by a North step the subpath of $P'$ which is to the Northeast of $s$, and we associate to $P'$ the corner of this new path which have $s$ as its East step. An example of the bijection is given in \fref{fig:thm_NAT}.

\begin{figure}[h!]
    \begin{center}
        \subfloat[A path P (dashed) and a path P'(blue) weakly below P, with $cc(P')=3$.]{\includegraphics[scale=.5]{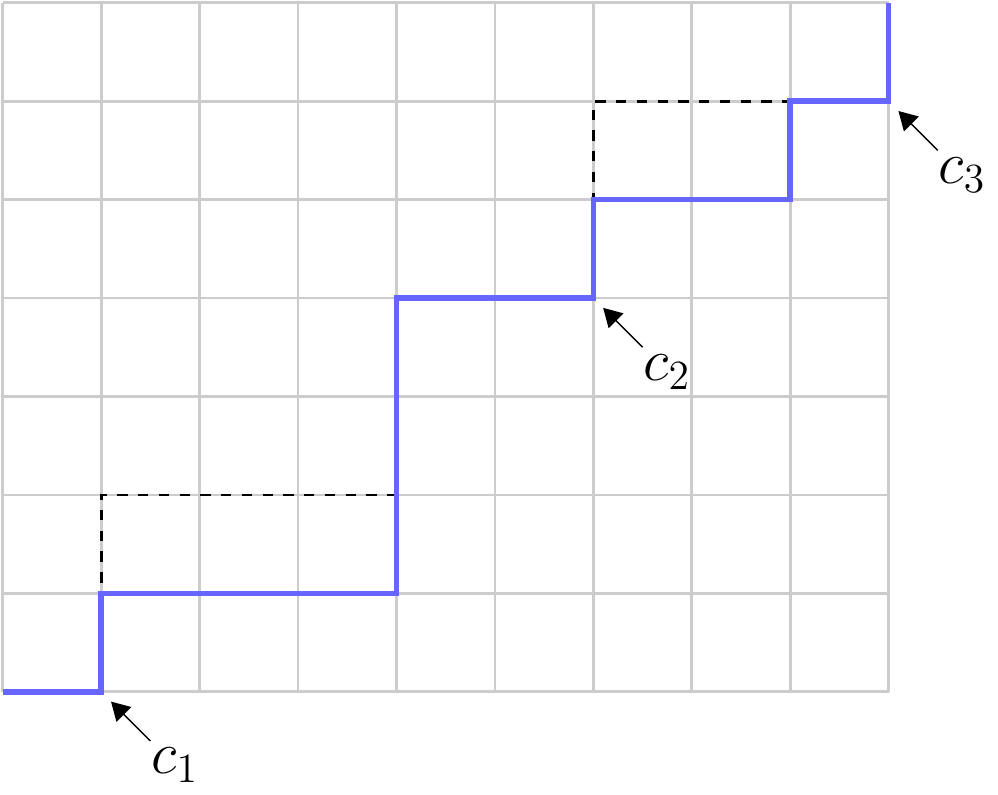}}
    \end{center}
    \begin{center}
        \subfloat[Path associated to $c_1$.]{\includegraphics[scale=.4]{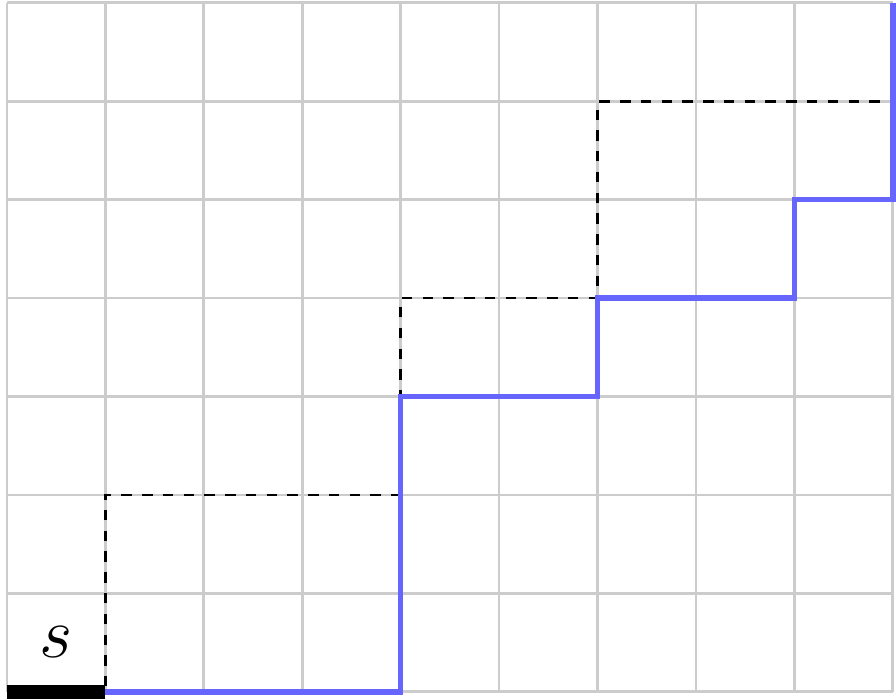}}
        \hfill
        \subfloat[Path associated to $c_2$.]{\includegraphics[scale=.4]{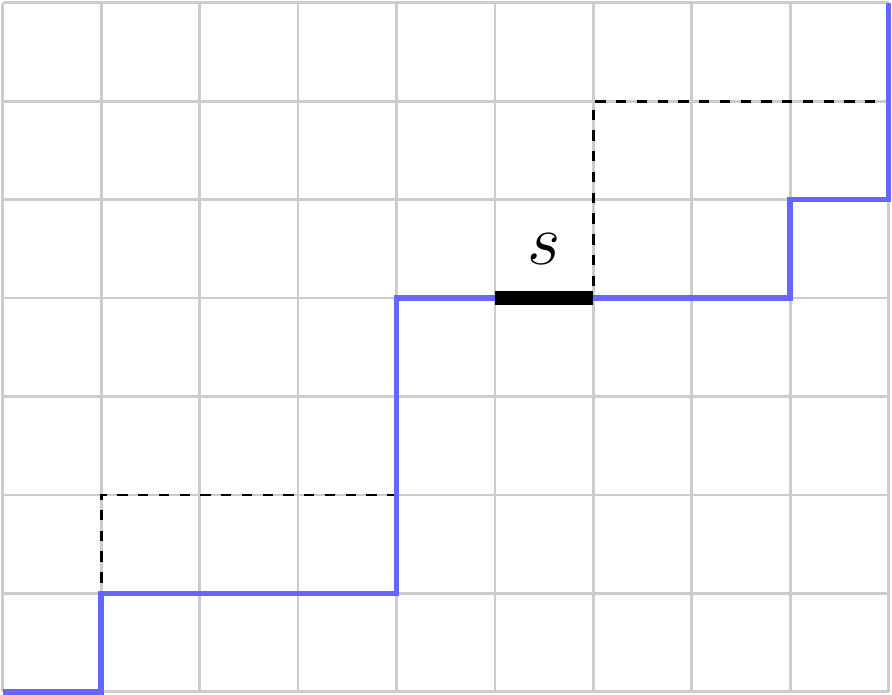}}
        \hfill
        \subfloat[Path associated to $c_3$.]{\includegraphics[scale=.4]{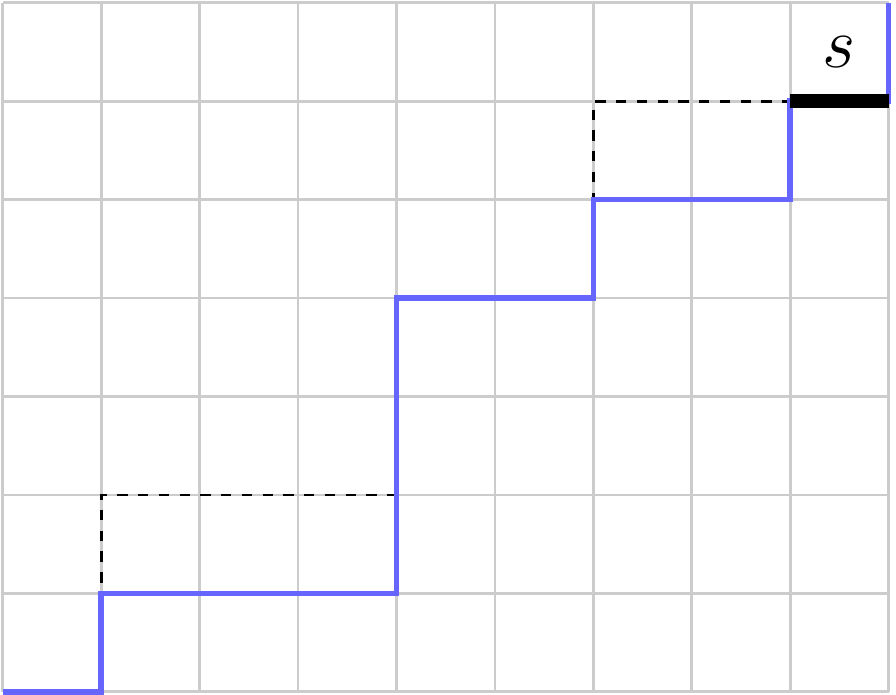}}
        \caption{Bijective proof of the \tref{thm:oc_NA}.}
        \label{fig:thm_NAT}
    \end{center}
\end{figure}
\section{The link between corners and the PASEP}\label{sec:conj}
As explained in the introduction, the equilibrium state of the PASEP can be described by tree-like tableaux. We define the weight $w(T)$ of a tree-like tableau as follows: $w(T)=q^{cr(T)}\alpha^{-tp(T)}\beta^{-lp(T)}$, where $cr(T)$, $tp(T)$ and $lp(T)$ correspond to the number of crossings, top points and left points of $T$, respectively (see \cite{avbona13} for the definitions of those statistics). Let $Z_n$ be the partition function $\sum_{T\in\T_{n+1}}w(T)$. Let $s$ be a word of size $n$ corresponding to a state in the PASEP, we denote by $\T_{n+1}^{\lambda(s)}$ the set of tree-like tableaux of size $n+1$ that projects to $s$. As shown in \cite{CorWil_Tableaux}, the stationary probability of $s$ is equal to:
$$\frac{\sum_{T\in\T_{n+1}^{\lambda(s)}}w(T)}{Z_n}.$$ Let $X(s)$ be the random variable counting the number of locations of a state $s$ where a particle may jump to the right or to the left. Let $T$ be a tree-like tableau that projects to $s$. An \emph{inner corner} of $T$ is a succession of a vertical border edge by an horizontal border edge, we denote by $ic(T)$ the number of inner corners of $T$. We have the immediate relation $ic(T)=c(T)-1$ where $c(T)$ represent the number of corners of $T$. With those notations, $X(s)=c(T)+ic(T)=2c(T)-1$. We can now compute the theoretical value of the expected value of $X$ in terms of tree-like tableaux.
\begin{flalign*}
\E(X)   &=\sum_{s\in\{\circ,\bullet\}^n}\Prob(s)\cdot X(s)\\
        &=\frac{1}{Z_n}\sum_{s\in\{\circ,\bullet\}^n}\sum_{T\in\T_{n+1}^{\lambda(s)}}w(T)(2c(T)-1)\\
        &=\frac{1}{Z_n}\sum_{T\in\T_{n+1}}w(T)(2c(T)-1)
\end{flalign*}
By computer exploration using Sage \cite{sage} we found the following conjecture.
\begin{conj}\label{conj:ec_conj}
The number of corners in $\T_n$ is $n!\times\frac{n+4}{6}$.
\end{conj}

In the case $\alpha=\beta=q=1$ and $\delta=\gamma=0$, then $w(T)=1$, hence this conjecture would imply that $$\E(X)=2\frac{(n+1)+4}{6}-1=\frac{n+2}{3}.$$ We were not able to adapt the proofs of \tref{thm:main_result} to prove this conjecture. To adapt the first proof, we should construct in average $\frac{n(n+4)}{6}$ corners in $\T_n$ from a tree-like tableau of size $n-1$. For the second one, we should control the behaviour of corners during an insertion, and translate this behaviour in terms of the code. We also give the symmetric version of the \coref{conj:ec_conj}.
\begin{conj}\label{conj:ec_sym_conj}
The number of corners in $\TS_{2n+1}$ is $2^n\times n\times\frac{4n+13}{12}$.
\end{conj}

\subparagraph*{Acknowledgements}
I am grateful to X. Viennot and O. Mandelshtam for helpful discussions which allowed me to find respectively the short proof of \tref{thm:main_result} and the proof of \tref{thm:oc_NA}.

This research was driven by computer exploration using the open-source software \texttt{Sage}~\cite{sage} and its algebraic combinatorics features developed by the \texttt{Sage-Combinat} community~\cite{Sage-Combinat}.

\bibliographystyle{alpha}
\bibliography{bibliographie}

\end{document}